\documentclass[12pt,leqno]{article}
\usepackage{makeidx}

\usepackage{amssymb,amsfonts,stmaryrd,mathrsfs}
\usepackage{amsmath,latexsym,amsthm}

\usepackage{hyperref}
\usepackage{manfnt}
\usepackage{verbatim}
\usepackage[applemac]{inputenc} % MacOS
\usepackage[cyr]{aeguill}
\usepackage[russian,francais,english]{babel}
\input cyracc.def

\usepackage%[disable]
{todonotes}

\usepackage{frcursive}

\usepackage{url}
\let\urlorig\url
\renewcommand{\url}[1]{%
  \begin{otherlanguage}{english}\urlorig{#1}\end{otherlanguage}%
}
\usepackage{graphicx}

\makeindex

\newtheorem{lemma}{Lemma}
\newtheorem{corollary}{Corollary}
\newtheorem{proposition}{Proposition}
\newtheorem{theorem}{Theorem}
\newtheorem{remark}{Remark}

\newcommand{\sC}{{\mathscr C}}

\newcommand{\vol}{{v}}

\newcommand{\bv}{{v}}

\newcommand{\R}{{\mathbb R}}

\newcommand{\Z}{{\mathbb Z}}

\newcommand{\capp}{{\kappa}}

\newcommand{\Ham}{\mathscr{H}}

\newcommand{\eps}{\varepsilon}

\begin{document}
\title{On the well-posedness of a quasi-linear \\Korteweg-de Vries equation}
\author{C. Mietka\thanks{Universit\'e de Lyon,
CNRS UMR 5208,
Universit\'e Lyon 1,
Institut Camille Jordan,
43 bd 11 novembre 1918;
F-69622 Villeurbanne cedex}}
\maketitle

\begin{abstract}  
The Korteweg-de Vries equation (KdV) and various generalized, most often \emph{semilinear} versions have been studied for about 50 years. Here, the focus is made on a \emph{quasi-linear} generalization of the KdV equation, which has a fairly general Hamiltonian structure.
This paper presents a local in time well-posedness result, that is existence and uniqueness of a solution and its continuity with respect to the initial data. The proof is based on the derivation of energy estimates, the major interest being the method used to get them. The goal is to make use of the structural properties of the equation, namely the skew-symmetry of the leading order term, and then to control subprincipal terms using suitable gauges as introduced by Lim \& Ponce (SIAM J. Math. Anal., 2002) and developed later by Kenig, Ponce \& Vega (Invent. Math., 2004) and S. Benzoni-Gavage, R. Danchin \& S. Descombes (Electron. J. Diff. Eq., 2006). The existence of a solution is obtained as a limit from regularized parabolic problems. Uniqueness and continuity with respect to the initial data are proven using a Bona-Smith regularization technique.
\end{abstract}


\tableofcontents

\section{Introduction and main result}

More than a century ago, D.J. Korteweg and G. de Vries proposed a model for unidirectional long water waves propagating in a channel. The so-called Korteweg-de Vries equation
$$\text{(KdV)} \quad v_t + v v_x + v_{xxx} = 0 \,,$$
in fact derived earlier by Boussinesq, drew a lot of attention in the 1960's, when it turned out that it was {\it completely integrable}, see e.g. one of the seminal papers by Gardner, Green, Kruskal \& Miura \cite{GGKM67,GardnerKDV} or the book by Ablowitz \cite{Ablowitz} for a modern overview. 

It was soon considered in generalized forms 
$$\text{(gKdV)} \quad v_t + p(v)_x + v_{xxx} = 0 \,.$$
In particular, the {\it modified} KdV equation 
$$\text{(mKdV)} \quad v_t + v^2v_x + v_{xxx} = 0 \,.$$
is also completely integrable. This is not the case for more general nonlinearities. Nevertheless, (KdV) and (gKdV) have been studied by analysts for about $50$ years. The state of the art regarding (KdV) is mainly due to Bona \& Smith \cite{Bona-Smith}, Kato \cite{kato75}, Kenig, Ponce \& Vega \cite{KPV93, KPV96} and Christ, Colliander \& Tao \cite{CCT03}. For (gKdV), it is due mainly to Kenig, Ponce \& Vega \cite{KPV93mkdv} and Colliander, Keel, Staffilani, Takaoka \& Tao \cite{CKSTT03}.

In this article, we consider a quasi-linear version of the Korteweg-de Vries equation, in which the dispersive term is not reduced to $v_{xxx}$. This equation is the most natural generalization of the abstract, Hamiltonian form of (gKdV), which reads
\begin{equation}
\label{eq:hamil}
v_t = \left(\delta \Ham[v]\right)_x
\end{equation}
with 
\begin{equation}
\label{eq:hamil2}
\Ham[v] = \frac{1}{2}v^2_x + f(v)\,,
\end{equation}
 and $f'(v) = -p(v)$. Our motivation for considering this generalization comes from the so-called Euler-Korteweg system, which involves an energy of the form \eqref{eq:hamil2}
where $\capp$ is not necessarily a constant. %The quasilinear KdV equation we consider is \eqref{eq:hamil} \eqref{eq:hamil2}.
The link with the Euler-Korteweg system is that their travelling waves share the same governing ODE, with additional connections in the stability of their periodic waves, see \cite{BMR1}. 

Our qKdV equation \eqref{eq:hamil} \eqref{eq:hamil2} reads in a more explicit form
\begin{equation}
\label{eq:qkdv}
\partial_t v + \partial_x (p(v)) + \partial_x\left( \sqrt{\capp(v)} \partial_x( \sqrt{\capp(v)} \partial_x v) \right) = 0 \,.
\end{equation}
Up to our knowledge, the Cauchy problem regarding this quasilinear equation has never been investigated. An apparently more general nonlinear KdV equation 
$$v_t + f(v_{xxx}, v_{xx}, v_x, v)=0\,,$$
 was studied by Craig, Kappeler, \& Strauss in \cite{CKS} and more recently by Linares, Ponce, \& Smith in \cite{LPS}. However, they use a monotonicity assumption on the nonlinearity, $\partial_{v_{xx}}f(v_{xxx}, v_{xx}, v_x, v) \leq 0$, which reads $\partial_x \capp(v) \leq 0$ for \eqref{eq:qkdv}. Unless $\capp$ is constant, their results are thus hardly applicable to \eqref{eq:qkdv}. Our approach is to use the {\it structure} of the equation instead of a monotonicity argument.

Of course, Eq. \eqref{eq:qkdv}, includes the semi-linear generalized KdV equations, where $\capp(v) = 1$ and the nonlinearity $p$ is polynomial. In particular $p(v) = \frac12 v^2$ corresponds to the classical KdV equation and $p(v) = \frac13 v^3$ to the modified KdV equation. As said before, these two cases are known to fall into the class of integrable equations. Remarkably enough, there is a non constant $\capp$ for which (qKdV) is completely integrable, namely $\capp(v) = \frac{\eps^2}{12}(v+a)^{-3}$, with $a \in \R$, $\eps >0$ and $p(v) = \frac{v^2}{2}$, see \cite{DBanff}. 

We focus on the local-in-time well-posedness of (qKdV) in Sobolev spaces, that is existence and uniqueness of a smooth solution $v$ given smooth initial data $v_0$, with continuity of the mapping $v_0 \mapsto v$. The most important part of the work is based on obtaining a priori estimates, using the skew-symmetric form of the leading order term and gauging techniques to control subprincipal remainders. The idea of using gauges for dispersive PDEs, introduced by Lim \& Ponce in \cite{limponce02} and developed by Kenig, Ponce \& Vega in \cite{KPV04} and later by S. Benzoni-Gavage, R. Danchin \& S. Descombes in \cite{BDD1d, BDDmultid}, is a fairly general method to deal with subprincipal terms.

In what follows, we consider an interval $I \subset \R$, and solutions of \eqref{eq:qkdv} are sought with values in a compact subset $J$ of $I$. The functions $\capp : I \to \R^{+*}$ and  $p : I \to \R$ are supposed to be smooth. 
We consider an integer $k$ and denote by $H^k(\R)$ the classical Sobolev space constructed on $L^2(\R)$ as
$$H^k(\R) = \{v \in L^2(\R)\,, \ \partial_x^l v \in  L^2(\R)\ for\ all\ | l | \leq k\}\,.$$

The inner product in $L^2(\R)$ will be denoted for all $u$ and $v$ by $\left<u|v \right>$.
Our main result is the following. 

\begin{theorem}\label{thm:qKdV}
Assume that $k\geq 4$. If $p=-f': I \subset \R \to \R$  is $\sC^{k+1}$ and $\capp:I \to \R^{+*}$ is $\sC^{k+2}$, then for all $\bv_0\in H^k(\R)$, the image of $\bv_0$ being in $J \subset \subset I$, there exists a time $T>0$ and a unique $\bv\in \sC(0,T;H^k(\R)) \cap \sC^1(0,T;H^{k-3}(\R))$ solution to \eqref{eq:qkdv} with initial data $v_0$. Moreover, $\bv_0\mapsto \bv$ maps continuously $H^k(\R)$ into $\sC(0,T;H^k(\R)) \cap \sC^1(0,T;H^{k-3}(\R))$.
\end{theorem}

\begin{remark}
From the above statement, by standard arguments one may then build maximal solutions and a close inspection of our estimates shows that the maximal lifespan is finite if and only if the $H^4$ norm of the solution blows up at the final time.
\end{remark}

The proof of this theorem is based on the derivation of a priori estimates for a regularized parabolic problem. We establish these a priori estimates on smooth solutions by taking advantage of the structure of Eq. \eqref{eq:qkdv}. More precisely, we use the skew-symmetry of its leading order term and then gauging techniques to deal with subprincipal remainders. Regarding the existence of solutions, we shall use a fourth order parabolic regularization of (qKdV) and pass to the limit. Uniqueness and continuity of the mapping $v_0 \mapsto v$ are proved by means of a priori estimates and a technique adapted from Bona \& Smith \cite{Bona-Smith}.  

We introduce the notations $a = p'= -f''$ and $\alpha = \sqrt{\capp}$. For convenience, we use the same notation for both functions $v \mapsto \alpha(v)$ and $(t,x)\mapsto (\alpha \circ v) (t,x)$. In particular, $\alpha'$ stands for the derivative of $\alpha(v)$ with respect to $v$ and $\partial_t \alpha$ and $\partial_x \alpha$ for the time and space derivatives of $\alpha \circ v$. We use the same convention for the function $a$ and all nonlinear functions of the dependent variable $v$, unless otherwise specified. To keep notations compact, and hopefully easier to read, we also omit all parentheses in operators. For example, the expression 
$$ \partial_x (\alpha(v) \partial_x( \alpha(v) \partial_xv))\,,$$
will just be denoted by
$$ \partial_x \alpha \partial_x \alpha \partial_xv\,.$$
With these conventions, Eq. (\ref{eq:qkdv}) becomes 
\begin{equation}
\label{eq:motion}
\vol_t + a \vol_x  + \partial_x \alpha \partial_x \alpha \partial_x \vol = 0\,,
\end{equation}
as far as smooth solutions are concerned.

%%%%%%%%%%%%%%%%%%%%%%%%%%%%%%%%%%%%%%%%%%%%%%%%%%%%%%%%%%%%%%%%%%%%%%%%%%%%%%%%%%%%%%%%%%%%%%%%%%%%%%%

\section{A priori energy estimates in Sobolev spaces}

In this section, we investigate {\it a priori} bounds in $H^k(\R)$ for smooth solutions of \eqref{eq:motion} with $k \geq 4$. Following ideas from \cite{limponce02, BDD1d, BDDmultid}, we shall make use of the structure of the equation and of gauges, in order to cancel out bad commutators. An ideal structure that would allow a direct computation of energy estimates is 
$$ v_t = \text{skew-symmetric terms} + \text{zero$^{th}$ order terms}\,.$$
Taking the inner product of this kind of equation with $v$, we see that the skew-symmetric terms cancel out and we readily find 
$$\frac{d}{d t}\|v\|_{L^2} \lesssim \|v\|_{L^2}\,.$$
Even if Eq. \eqref{eq:motion} has not exactly this ideal structure, the presence of first order terms can also be handled, up to an integration by parts, not directly in $L^2(\R)$ but in higher order Sobolev spaces. If we intend to get energy estimates in higher order Sobolev spaces than $L^2(\R)$, we had better be sure that this structure is preserved when differentiating the equation to avoid derivatives loss. In what follows, we adapt a method from Lim \& Ponce \cite{limponce02}, and use weighted Sobolev spaces, with weights also called {\it gauges}, to transform our equation into a convenient structure and derive a priori estimates without loss of derivatives.

\subsection{Conservation of the  skew-symmetric structure}
Let us first focus on the leading order term in \eqref{eq:motion}. When looking for an a priori estimate in $L^2(\R)$, we compute the time derivative of $\|v\|_{L^2}$ by taking the inner product of the equation with $v$. On the skew-symmetric leading order term, an integration by parts yields 
$$\left< \partial_x \alpha \partial_x \alpha \partial_x v | v \right> = - \left< \partial_x \left(\alpha \partial_x v\right) | \alpha \partial_x v \right>  =  0\,. $$ 
If we apply the differential operator $\partial_x$ to Eq. \eqref{eq:motion}, we lose the skew-symmetry property and the corresponding cancellation. Our aim is to preserve this cancellation throughout the entire differentiation process. For this purpose, we consider the weighted quantity 
\begin{equation}
\forall k \geq 0 \,, \quad \vol_k = \left(\alpha(v) \partial_x\right)^k \vol \,,
\end{equation}
instead of $\partial^k_x v$.
This quantity is well defined if $v$ is smooth enough. If we apply the formal operator $ \left(\partial_x(\alpha \  \cdot \ )\right)^k$ to Eq. \eqref{eq:motion}, we see that the higher order terms have the same form as in the original equation. They read $ \partial_x \alpha \partial_x \alpha \partial_x v_k$, and thus cancel out in the inner product with $v_k$.

Following a definition in \cite{Serre} (\S 3.6), we call {\it weight} the total number of space derivatives in a monomial expression involving a function and its own derivatives. As we will see, when we compute derivatives by applying repeatedly the operator $\partial_x(\alpha \cdot)$ to Eq. \eqref{eq:motion}, the coefficients of the remainders we will have to deal with will be products of polynomial functions of $v$ and its derivatives with functions of $\alpha(v)$ or $a(v)$ and their derivatives. As a result, they are polynomials in the variables $v_k$ with terms multiplied by functions of $v$ but we extend the definition of weight to this kind of non-polynomial functions. The weight is merely the total number of derivatives. 
\begin{proposition}
\label{prop:structure}
For a smooth solution of Eq. \eqref{eq:motion}, if we denote $v_k = \left(\alpha(v) \partial_x\right)^k v$, then the equation satisfied by $v_k$ is of the form
\begin{equation}
\label{eq:motion_k}
\partial_t \vol_k + a \partial_x \vol_k + \partial_x \alpha \partial_x \alpha \partial_x \vol_k = f_k \partial^2_x \vol_k + g_k \partial_x \vol_k + h_k \,,
\end{equation}
where 
\begin{itemize}
\item $f_k= f_k(v,v_1)$ is of (homogeneous) weight 1.
\item $g_k= g_k(v,v_1,v_2)$ is of (homogeneous) weight 2.
\item $h_k= h_k(v,v_1, \cdots , v_k)$ consists of two terms, one of (homogeneous) weight $k+1$ and another of (homogeneous) weight $k+3$.
\end{itemize}
\end{proposition}
In Eq. \eqref{eq:motion_k}, we kept a skew-symmetric leading order term in the left hand side on purpose and we gathered commutator terms in the right hand side. These are subprincipal terms because commutator terms between two differential operators of order respectively $p_1$ and $p_2$ are of order $p_1 + p_2 -1$.
\begin{proof}
Eq. \eqref{eq:motion_k} is obtained by repeatedly applying the differential operator $\partial_x (\alpha \cdot )$, which preserves the form of the higher order terms. Indeed, 
$$\partial_x \alpha \partial_x \alpha \partial_x \alpha \partial_x v_k =  \partial_x \alpha \partial_x \alpha \partial_x v_{k+1}\,,$$
where the parentheses are omitted. We apply the differential operator $\partial_x(\alpha \cdot )$ and find the equation satisfied by $v_1 = \alpha \partial_x v$. We have 
$$ \partial_x(\alpha \partial_t v ) = \partial_t v_1\,, \quad \partial_x(\alpha a \partial_x v) = a \partial_x v_1 + a' \alpha \left(\partial_x v\right)^2\,, $$ 
so that the equation satisfied by $v_1$ is 
$$ \partial_t v_1 + a\partial_x v_1 + \partial_x \alpha\partial_x \alpha \partial_x v_1 = - \frac{a'}{\alpha}v^2_1\,,$$
hence $f_1= 0$, $g_1 = 0$ and $h_1 =  - \frac{a'(v)}{\alpha(v)}v^2_1$.
For $k = 2$, we use the equation on $v_1$ and apply again the operator $\partial_x(\alpha \cdot )$. 
Then, we find 
$$\partial_t v_2 + a\partial_x v_2 + \partial_x \alpha\partial_x \alpha \partial_x v_2  = -\alpha'v_1 \partial^2_x v_2 +\left( \frac{\alpha'}{\alpha}v_2 - \frac{\alpha'^2}{\alpha^2}v^2_1 \right)\partial_x v_2 +  \left(\frac{\alpha' a}{\alpha} - \frac{a''}{\alpha}  \right)v^3_1 - \frac{3a'}{\alpha}v_1 v_2   \,,$$
hence $f_2 = \alpha' v_1$, $g_2 =  \frac{\alpha'}{\alpha}v_2 - \frac{\alpha'^2}{\alpha^2}v^2_1 $ and 
$$h_2 =  \left(\frac{\alpha' a'}{\alpha^2} - \frac{a''}{\alpha}  \right)v^3_1 - \frac{3a'}{\alpha}v_1 v_2 \,.$$ 
More generally, $f_k$, $g_k$, $h_k$ are computed by induction for $k \geq 2$. We have 
$$\partial_x(\alpha \partial_t v_k) = \partial_t v_{k+1} - \alpha' \left[  v_t  \partial_x v_k - v_x(v_k)_t\right]\,,$$ 
where we can use Eq. \eqref{eq:motion} and \eqref{eq:motion_k} to write 
$$\partial_t v_k =  - \frac{a}{\alpha} v_{k+1} - \alpha \partial^2_x v_{k+1} - \frac{\alpha'}{\alpha} v_1 \partial_x v_{k+1}+ \frac{f_k}{\alpha} \partial_x v_{k+1} - \frac{\alpha'}{\alpha^2} f_k v_1v_{k+1} + \frac{g_k}{\alpha} v_{k+1} + h_k\,,$$
and
$$v_t = -\frac{a}{\alpha}v_1 - \partial_x v_2 \,.$$
Using that
$$\partial_x(\alpha a \partial_x v_k) = a \partial_x v_{k+1} + a' \alpha \partial_x v \partial_x v_k = a \partial_x v_{k+1} + \frac{a'}{\alpha} v_1 v_{k+1}\,,$$
and	
\begin{equation*}
\begin{array}{rcl}
\partial_x \left(\alpha \left[ f_k \partial^2_x \vol_k + g_k \partial_x \vol_k + h_k \right] \right) &=& f_k \partial^2_x v_{k+1}\\
&+& \left(g_k + \partial_x f_k -  \frac{\alpha'}{\alpha^2} f_k v_1\right) \partial_x v_{k+1} \\
&+& \left( \partial_x g_k - \partial_x\left(\frac{\alpha'}{\alpha^2}f_k v_1\right)  \right)v_{k+1} + \partial_x(\alpha h_k)\,.
\end{array}
\end{equation*}
we find that 
$$\partial_t v_{k+1} + a \partial_x v_{k+1} + \partial_x \alpha \partial_x \alpha \partial_x v_{k+1} = f_{k+1} \partial^2_x v_{k+1} + g_{k+1} \partial_x v_{k+1} + h_{k+1} \,,$$
with 
\begin{equation}
\label{eq:induction}
\begin{array}{r c l}
f_{k+1} &=& f_k + \alpha' v_1 \\[1ex]
g_{k+1} &=& g_k + \partial_x f_k - \frac{2\alpha'}{\alpha^2} f_k v_1 + \frac{\alpha'^2}{\alpha^2} v^2_1\\[1ex]
h_{k+1} &=& \alpha \partial_x h_k+ \left( \partial_x g_k - \partial_x\left( \frac{\alpha'}{\alpha^2}f_kv_1\right)  \right)v_{k+1} \\[1ex]
 && \quad \quad \quad  \quad \quad \quad  \quad \quad \quad  - \frac{a'}{\alpha}v_1v_{k+1}+ \frac{\alpha' }{\alpha} v_{k+1}\left( \frac{\alpha'}{\alpha^2}f_kv^2_1 -\frac{g_k}{\alpha}v_1 - \partial_x v_2 \right) 
\end{array}
\end{equation}

Note that these induction relations are valid for $k \geq 2$. Indeed, the last term in \eqref{eq:induction} in the formula for $h_k$, namely $ -\frac{\alpha'}{\alpha} v_{k+1} \partial_x v_2$ involves three derivatives on $v$ at least. For $k = 1$, this term should be gathered with the ones which define $g_{k+1} = g_2$. This is the reason why we detailed the cases $k=1$ and $k=2$ in the beginning of the proof. 
From the first induction, we deduce that for all $k\geq 1$
\begin{equation}
\label{eq:f_k}
f_k = (k-1)\alpha' v_1= (k-1)\alpha\partial_x\alpha\,.
\end{equation}
This explicit expression will be useful in what follows and justifies that the coefficient $f_k$ depends only on $v$ and $v_1$ and is a polynomial in $v_1$ with weight one multiplied by a bounded function of $v$, namely $v \mapsto (k-1)\alpha'(v)$. We do not actually need the exact expression for $g_k$ and $h_k$. Instead, we analyse the number of derivatives they contain, that is their {\it weight}, and their general form in terms of the variables $v, v_1, \cdots, v_k$. 

For $k=2$, we have the explicit expression of $g_2$ which is of weight $2$. For $k \geq2$ if $g_k$ is of weight $2$ then so is $g_{k+1}$ because
$$g_{k+1} - g_k = \partial_x f_k - \frac{2\alpha'}{\alpha^2} f_k v_1 + \frac{\alpha'^2}{\alpha^2} v^2_1$$
is of weight $2$, since $f_k$ is of weight $1$. 

Finally, we shall prove by induction that $h_k$ has the general form 
\begin{equation}
\label{eq:general_form}
h_k(v,v_1, \cdots , v_k) =  \sum_i \beta_i(v) \prod_{j=1}^k \partial^{\gamma_j}_x v \,,
\end{equation}
where the $\beta_i$ involves derivatives of both functions $v \mapsto \alpha(v)$ and $ v \mapsto a(v)$ and for all $j$, $\gamma_j \leq k$ and the weight of each of the terms, namely  $\sum_j \gamma_j$ is either $k+1$ or $k+3$. Note that in the general form \eqref{eq:general_form}, the $\gamma_j$ are not necessarily distinct. For $k=2$ we have
$$h_2 =  \left(\frac{\alpha' a'}{\alpha^2} - \frac{a''}{\alpha}  \right)v^3_1 - \frac{3a'}{\alpha}v_1 v_2 \,,$$ 
which is of weight $3$. It is compatible with the general form \eqref{eq:general_form}.
For $k \geq 2$, we have from Eq. \eqref{eq:induction}
$$h_{k+1} = \alpha \partial_x h_k - \frac{a'}{\alpha}v_1v_{k+1} + r_k\,,$$
where $r_k$ gathers terms of weight $k+4$, as $g_k$ is of weight $2$ and $f_k$ is of weight $1$. The term $- \frac{a'}{\alpha}v_1v_{k+1}$ is of weight $k+2$. Now, from the form \eqref{eq:general_form} and using the chain rule and the product rule, we write 
$$\alpha \partial_x h_k = \alpha(v) \partial_x \left( \sum_i \beta_i(v) \prod_{j=1}^k \partial^{\gamma_j}_x v \right) = \sum_i \tilde \beta_i(v) \prod_{j=1}^k \partial^{\tilde \gamma_j}_x v \,,$$
where  the $\tilde \beta_i$ also involve derivatives of both functions $v \mapsto \alpha(v)$ and $ v \mapsto a(v)$ and for all $j$, $\tilde \gamma_j \leq k+1$ and the weight of each of the terms, namely $\sum_j \tilde \gamma_j$ is either $k+2$ or $k+4$. We see that the term $\alpha \partial_x h_k $ has the same general form as $h_k$, but with terms of weight $k+2$ or $k+4$. Then, if $h_k$ has the general form \eqref{eq:general_form}, so it is for $h_{k+1}$, with weight raised by one.  
\end{proof}

\begin{remark}
Another way to see that $h_k$ has the general form \eqref{eq:general_form} is to see that it is made of terms of the form 
$$\left((\alpha \circ v) \partial_x\right)^k (f \circ v) \,,$$
the function $f$ being a combination of the functions $\alpha$ or $a$ and their derivatives. Combining the Fa\`a di Bruno formula, which generalizes the chain rule, and the product rule, we see that the term $h_k$ consists of a polynomial expression on $v$ and its derivatives up to $v_k$, each term of the polynomial being multiplied by a bounded function of $v$.

\end{remark}

The structure of Eq. \eqref{eq:motion_k} is
$$ \partial_t v_k = \text{skew-symmetric terms} + \text{terms of order at most }2 \,.$$
In what follows, we will show how we can use {\it gauges} to reduce the order of these remainders. We shall start with the control of the zero order term. Then we will find an estimate for the first order term, without gauge fortunately. Finally, we will show how we can use a {\it suitable gauge} to control the second order term. 
\begin{remark}
We will have to prove a norm equivalence between the Sobolev norm $\| v \|_{H^k}$ and the weighted norm $\| \vol_k \|_{L^2}$. This property will be checked at the end of this section. 
\end{remark}

\begin{remark}
\label{rq:estimates}
To obtain Theorem~1 with $k=4$, one may think about a slightly different strategy. It is based on estimating not  $(\alpha\circ v \partial_x)^4v$ --- which we do by introducing some gauge --- but $(\alpha\circ v \partial_x)\partial_tv$. Since $\|(\alpha\circ v \partial_x)\partial_tv\|_{L^2}^2$ is the leading order part of $\langle \partial_tv, \delta^2\Ham[v]\partial_t v$, the subprincipal terms appearing in the latter computation may dealt with in a gaugeless way. See Remark \ref{rq:estimates2} for a continuation of this comparison.
\end{remark}

\subsection{Gauge estimates for subprincipal remainders}

Let us first focus on the zero order term $h_k$. 
We first prove the following lemma, which is an extension of Lemma 3.6.2 in \cite{Serre} to expression of order higher than one and adapted to our particular functions, that is polynomial functions with terms multiplied by bounded functions of $v$. 
\begin{lemma}
\label{lemma:hk}
Let $Q(\partial_x v, \cdots, \partial^k_x v)$ be an homogeneous polynomial of weight $q \in [k,2k)$. There exists a smooth function $C_{q,k}$ such that, for all $v \in H^k(\R)$, 
$$ \| Q(\partial_x v, \cdots, \partial^k_x v) \|_{L^2} \leq C_{q,k}( \|\partial_xv \|_{L^{\infty}}, \cdots, \| \partial^{(q-k)}_x v \|_{L^{\infty}}) \|v\|_{H^k} $$
\end{lemma}
\begin{proof}
The proof is based on the Gagliardo-Nirenberg inequality. First, note that if $v \in H^k(\R)$ and $q< 2k$, then for all $p \leq q- k < k$, $\partial^p_x v \in H^1(\R) \hookrightarrow L^{\infty}(\R)$.  
By the triangular inequality, it is sufficient to focus on a monomial expression of the form 
$$ Q(\partial_x v, \cdots, \partial^k_x v)  = \prod_{j=1}^k \partial^{\gamma_j}_x v \,,$$
with for all $1 \leq j \leq k$, $\gamma_j  \leq k$ satisfy $ \sum_{j=1}^k \gamma_j = q$.
If there are some $j$ such that $\gamma_j \leq q - k$, then we can estimate the corresponding factors in $L^{\infty}(\R)$. Now, if all the remaining factors satisfy $q - k < \gamma_j < k$, we can conclude by using the Gagliardo-Nirenberg inequality. We choose a particular $l$ and write, for $j \neq l$
$$\| \partial^{\gamma_j}_x v \|_{L^{\infty}}  \lesssim \|\partial^k_x v \|^{\theta_j}_{L^2} \| \partial^{(q-k)}_x v\|^{1-\theta_j}_{L^{\infty}} \ , \textrm{where} \ \theta_j = \frac{\gamma_j - (q-k)}{k - (q-k) -1/2} \ ,$$
and for $j = l$
$$\| \partial^{\gamma_{l}}_x v \|_{L^{2}}  \lesssim \|\partial^k_x v \|^{\theta_{l}}_{L^2} \| \partial^{(q-k)}_x v\|^{1-\theta_{l}}_{L^{\infty}} \ , \textrm{where} \  \theta_{l} = \frac{\gamma_{l} - (q-k)-1/2}{k - (q-k) -1/2} \ .$$
With these relations, we have 
$$\theta = \sum_{\substack{j=1\\ q-k < \gamma_j \leq k}}^k \theta_j  \leq 1 \ ,$$ 
and then we are able to write 
$$ \| Q(\partial_x v, \cdots, \partial^k_x v) \|_{L^2} \leq C_{q,k}( \|\partial_x v \|_{L^{\infty}}, \cdots, \| \partial^{(q-k)}_x v \|_{L^{\infty}})\|\partial^k_x v \|^{\theta}_{L^2} \| \partial^{(q-k)}_x v\|^{1-\theta}_{L^{\infty}} \,.$$
Finally, if there exists $j_0$ such that $\gamma_{j_0} = k$, the corresponding factor belongs to $L^2(\R)$ and all other $j$ satisfies $\gamma_j \leq q-k$. Then, for all $j \neq j_0$, we have $\partial^{\gamma_j}_x v \in L^{\infty}(\R)$.
Combining all previous cases, and using the Sobolev embedding $H^1(\R) \hookrightarrow L^{\infty}(\R)$, we can write 
$$ \| Q(\partial_x v, \cdots, \partial^k_x v) \|_{L^2} \leq C_{q,k}( \|\partial_x v \|_{L^{\infty}}, \cdots, \| \partial^{(q-k)}_x v \|_{L^{\infty}})\| v \|_{H^k} \,.$$ 
\end{proof}

Recall that $h_k$ contains two terms of weight $q = k+1$ and $q = k+3$. We deduced in Proposition \ref{prop:structure} the general form 
$$h_k(v,v_1, \cdots , v_k) =  \sum_i \beta_i(v) \prod_{j=1}^k \partial^{\gamma_j}_x v \,,$$
where the $\beta_i$ are bounded functions of $v$ on $J$. Then, Lemma \ref{lemma:hk} gives us the following estimate 
$$ \| h_k(v,v_1, \cdots , v_k) \|_{L^2} \leq C_{k}( \|v \|_{L^{\infty}}, \cdots, \| \partial^{3}_x v \|_{L^{\infty}})\| v \|_{H^k} \,.$$

Let us now focus on the first order terms $g_k\partial_x v_k$. This term brings no trouble at all as we can regroup it with the first order term in Eq. \eqref{eq:motion_k}
$$ \left[a(v) - g_k(v,v_1,v_2) \right]\partial_x v_k \ .$$
When taking the inner product with $v_k$, we can estimate 
\begin{equation*}
\begin{array} {r l}
\left< \left(a - g_k\right) \partial_x v_k | v_k \right>  \leq &\| \partial_x \left( a - g_k \right) \|_{L^{\infty}}  \|v_k \|^2_{L^2} \\[1ex]
 \leq & C_k(\|v\|_{L^{\infty}},\|v_1\|_{L^{\infty}},\|v_2\|_{L^{\infty}}, \|v_3\|_{L^{\infty}})\| v_k \|^2_{L^2} 
\end{array}
\end{equation*}

This leaves us with the only one remaining term $f_k\partial^2_x v_k$. This one cannot be estimated as the previous ones because it contains too many derivatives. The method we present here consists in using \textit{gauges} in the equation, following ideas from \cite{limponce02}.  

Formally and in all generality, what we call a gauge is a general differential operator with unknown variable coefficients that cancels out `bad' commutator terms when applied to the equation. The key property is that the commutator of two differential operators of order respectively $p_1$ and $p_2$ is of order $p_1 + p_2 -1$. In our case, the equation is of leading order three, with a priori two subprincipal terms we wish to reduce to order zero. We could define a gauge of the type 
$$ \phi = \text{zero$^{th}$ term} + \text{order $(-1)$ term} $$ 
and apply it to our equation. Doing so, the two commutators with the leading order term would be of second and first order. These two new terms can be gathered with the existing ones, and, in practice, one can choose the coefficients of the gauge as solutions of ODEs, to cancel (or at least control) the subprincipal terms. 

In our situation, we have already shown that the first order term can be controlled without this technique. Using this fact, we will define a particular gauge as a function  
\begin{equation}
\label{eq:phi}
\phi_k : (v, \cdots, v_k) \mapsto \phi_k(v, \cdots , v_k)\,,
\end{equation}
to cancel the second order term $f_k\partial^2_x v_k$. As we will see in the computation though, we will have to check that the arising first order term can be bounded.

Now, we multiply Eq. \eqref{eq:motion_k} by $\phi^2_k$
$$\phi_k \partial_t( \phi_k \vol_k) + \phi_k a  \partial_x (\phi_k \vol_k) + \partial_x \phi_k \alpha \partial_x \phi_k \alpha \partial_x \vol_k + R_k = \phi_k \phi_k f_k \partial^2_x \vol_k + \phi_k \phi_k g_k \partial_x \vol_k + \phi_k \phi_k h_k \ , $$
where the four remaining terms are gathered in  
$$ R_k =  \phi_k a \left[ \phi_k, \partial_x \right] v_k +  \phi_k \left[ \phi_k, \partial_t \right] v_k + \left[ \phi_k \phi_k , \partial_x(\alpha \cdot) \right]\partial_x \alpha \partial_x v_k + \partial_x \alpha \phi_k \left[ \phi_k,  \partial_x \right]\alpha \partial_x v_k \ .$$ 
We expect the first two terms to be bounded in $L^2(\R)$ because they are of order zero. This will be checked when we find bounds on the function $\phi_k$. First, we compute the commutators
\begin{equation}
\label{eq:commute}
\begin{array}{l}
\left[ \phi^2_k, \partial_x(\alpha \cdot) \right] =   - 2\alpha \phi_k \partial_x\phi_k \,, \\
\left[ \phi_k,  \partial_x \right ] =   - \partial_x \phi_k\,, \\
\left[ \phi_k,  \partial_t \right ] =   - \partial_t \phi_k\,. 
\end{array}
\end{equation}
With these relations we are able to compute the last two terms in $R_k$ 
\begin{equation*}
\begin{array}{l}
 \left[ \phi^*_k \phi_k , \partial_x(\alpha \cdot) \right]\partial_x \alpha \partial_x v_k = -  2 \alpha^2 \phi_k(\partial_x\phi_k )\partial^2_x v_k - 2\alpha (\partial_x \alpha) \phi_k( \partial_x\phi_k )\partial_x vk \ , \\
\partial_x \alpha \phi^*_k \left[ \phi_k,  \partial_x \right]\alpha \partial_x v_k =  - \alpha^2 \phi_k (\partial_x \phi_k) \partial^2_x v_k - \partial_x \left( \alpha^2 \phi_k (\partial_x \phi_k)\right) \partial_x v_k \ .
\end{array}
\end{equation*}
As expected, these terms coming from the commutators are also subprincipal terms. We gather the coefficients of the second order terms and find an ODE on the function $\phi_k$.
\begin{equation}
\label{eq:edos}
3 \alpha^2 (\partial_x\phi_k)  + \phi_k f_k = 0\,.
\end{equation}
\begin{remark}
\label{rq:estimates2}
To proceed with the comment of Remark \ref{rq:estimates}, we observe that $\|\phi_4\circ v (\alpha\circ v \partial_x)^4v\|_{L^2}=\|\partial_x(\alpha\circ v \partial_x)^3v\|_{L^2}$ which differs from $\|(\alpha\circ v \partial_x)\partial_tv\|_{L^2}$ --- that is essentially $\|\alpha\circ v \partial_x^2(\alpha\circ v \partial_x)^2v\|_{L^2}$ --- only in some immaterial way. The advantage of using the gauge strategy to determine a correct functional --- instead of deducing it directly from the Hamiltonian structure --- is that it naturally generalizes to differentiation by any number of derivative, as we have just shown.
\end{remark}
We recall that this process added some commutator terms of first order terms we have to control a posteriori. If we can prove the existence of $\phi_k$ satisfying the previous ODE and belonging to a suitable space, here say $W^{3,\infty}(\R)$ for example, then we can get our a priori energy estimate by the same argument on first order terms as the one presented above. 
\begin{remark}
We also need to prove a norm equivalence between the usual $L^2$ norm and some $L^2$ norm involving $\phi_k$. This will be done at the end of this section. 
\end{remark}
From Eq. \ref{eq:f_k}, we come back to the ODE \eqref{eq:edos} and easily find
\begin{equation}
\phi_k(v) = \alpha(v)^{- \frac{k-1}{3}} \,.
\end{equation} 
Now, with the regularity properties of the function $v \mapsto \alpha(v)$, we directly get that $\phi_k$ is bounded from above and away from zero. Moreover, for all $0 \leq l \leq 3$, there exists a constant $C_k$ such that 
$$ \| \partial^l_x \phi_k(v) \|_{L^{\infty}} \leq C_k( \| v \|_{W^{3,\infty}})\,,$$
and 
$$ \| \partial_t \phi_k(v) \|_{L^{\infty}} \leq C_k(\|v\|_{W^{3,\infty}})\,.$$ 
Coming back to the first order terms, we check that 
$$ \|\partial_x \left[ a(v) - g_k - 2\alpha (\partial_x \alpha) \phi_k( \partial_x\phi_k ) - \partial_x \left( \alpha^2 \phi_k (\partial_x \phi_k)\right)\right] \|_{L^{\infty}} \leq C_k(\|v\|_{W^{3,\infty}}) \,.$$

Let us sum up {\it formally} what we obtained until now. At order $k > 0$, $v_k$ satisfies Eq. \eqref{eq:motion_k}. We define a gauge 
$$\phi_k(v) = \alpha(v)^{-(k-1)/3}\,,$$
and multiply Eq. \eqref{eq:motion_k} by $\phi^2_k$. This operation yields the equation satisfied by the quantity $\phi_k v_k$ 
$$\phi_k \partial_t( \phi_k \vol_k) + \partial_x \phi_k \alpha \partial_x \phi_k \alpha \partial_x \vol_k = \phi^2_k \widetilde R_k(v) \,,$$
where the function $\widetilde R_k$ gathers all the subprincipal terms we are now able to estimate 
$$ \| \phi_k \widetilde R_k(v) \|_{L^2} \leq C_k(\|v\|_{W^{3,\infty}}) \| \phi_k v_k \|_{L^2} \,.$$
This is a formal computation and we do not really get this ideal last estimate. In practice, it is not the norm $\| \phi_kv_k \|_{L^2}$ that appears but a combination of it and some norms $\|v\|_{H^k}$ with or without the gauge $\phi_k$. We need to prove norm equivalences between the weighted norms we introduced up to this point.

\subsection{Weighted norms equivalences}
As mentioned before, the final energy estimate cannot be obtained if we do not have some norm equivalence on the quantities we are working with. More precisely, we prove the following lemma
\begin{lemma} 
\label{lemma:norm}
Consider an integer $k \geq 1$, let $J$ be a compact subset of $I \subset \R$. Let $\phi_k = \alpha^{-(k-1)/3}$ with $\alpha : I \to \R^{+*}$ of class $\sC^{k+2}$. On the one hand, there exists a constant $c_k$ depending only $J$ such that, for all function $v \in H^k(\R)$ satisfying $v(t,x) \in J$  for all $(t,x) \in \R^+ \times \R$, then 
\begin{equation}
\label{eq:ineq1}
\frac{1}{c_k} \|v \|_{L^2} \leq \|\phi_k(v) v \|_{L^2} \leq c_k \| v \|_{L^2} \ .
\end{equation}
On the other hand, if we denote $v_k = (\alpha(v)\partial_x)^k v$, there exist constants $c'_k$ and $C_{k-1}$ depending only on a constant $\rho > 0$ and $J$, such that for all function $v \in H^k(\R)$ with $v(t,x) \in J$ for all $(t,x) \in \R^+ \times \R$ and $\|v\|_{W^{1,\infty}} \leq \rho$, then
\begin{equation} 
\label{eq:ineq2}
\frac{1}{c'_{k}} \|v\|^2_{H^k} \leq C_{k-1}\| v\|^2_{H^{k-1}} + \| \phi_k v_k \|^2_{L^2} \leq c'_{k} \| v \|^2_{H^k} \ .
\end{equation}
\end{lemma}
\begin{proof}
The first inequalities are a direct consequence of the fact that the function $\alpha$ is bounded from above and below as $J$ is compact.
Regarding the second norm equivalence, we use the same scheme of proof as in Lemma \ref{lemma:hk} with a weight $q = k$. 
Using the definition of $v_k$, we can write 
$$ v_k - \alpha^k(v)\partial^k_x v= \sum_i \beta_i(v) \prod_{j=1}^{k} \partial^{\gamma_j}_x v\,,$$
where the $\beta_i$ are bounded functions on $J$, $\sum^k_{j=1} \gamma_j = k$ and for all $j$, $\gamma_j \leq k-1$.
The leading order term in $v_k$ satisfies 
$$ \frac{1}{c_k} \|\partial^k_x v\|_{L^2} \leq \|\alpha^k(v)\partial^k_x v  \|_{L^2} \leq c_k\| \partial^k_x v \|_{L^2} \ .$$
By Lemma \ref{lemma:hk}, the remaining terms are bounded by 
$$\left\| v_k - \alpha^k\partial^k_x v \right\|^2_{L^2} \le2q C_{k-1}(\|v\|_{W^{1,\infty}})\| v \|^2_{H^{k-1}} \ .$$
Then, we get
$$ \frac{1}{c'_k} \|\partial^k_x v\|^2_{L^2}  - C_{k-1}\| v \|^2_{H^{k-1}}\leq \| v_k \|^2_{L^2} \leq c'_k\| \partial^k_x v \|^2_{L^2} + C_{k-1}\| v \|^2_{H^{k-1}} \,,$$
rearranged as 
$$ \frac{1}{c'_{k}} \| \partial^k_xv \|^2_{L^2} \leq C_{k-1}\| v\|^2_{H^{k-1}} + \| v_k \|^2_{L^2} \leq c'_{k} \| v \|^2_{H^k} \ .$$
Using the inequality \eqref{eq:ineq1} and without renaming the constants already written for convenience, we finally obtain
$$ \frac{1}{c'_{k}} \| v \|^2_{H^k} \leq C_{k-1}\| v\|^2_{H^{k-1}} + \| \phi_k v_k \|^2_{L^2} \leq c'_{k} \| v \|^2_{H^k} \ .$$
\end{proof}

Then for all $v \in H^k(\R)$, we consider the weighted norm $|\cdot |_k$ defined recursively by 
\begin{equation} 
\label{eq:norm_weight}
\left\{
\begin{array}{l}
\left| v \right|^2_k = \| \phi_k v_k\|^2_{L^2} + C'_{k-1}|v|^2_{k-1}  \,, \text{ for } k \geq 1 \,,\\[1ex]
\left| v \right|_0 = \|v\|_{L^2}\,.
\end{array}
\right.
\end{equation}
where the constant $C'_{k-1}$ is determined by induction. This definition leads to the following proposition. 
\begin{proposition}
\label{prop:equiv}
For all integer $s \geq 1$, the weighted norm $\left| \cdot \right|_s $ defined by \eqref{eq:norm_weight} is equivalent to the $H^s$ norm. More precisely,  there exists a constant $c_s$ depending only on $\rho > 0$ and $J$ such that for all $v \in H^s(\R)$ with, for all $(t,x) \in \R^+ \times \R$, $v(t,x) \in J$ and $\|v\|_{W^{1,\infty}} \leq \rho$, then
$$ \frac{1}{c_{s}} \| v \|^2_{H^s}\leq \left| v \right|^2_s \leq c_{s}\| v \|^2_{H^s} \,.$$
\end{proposition}
\begin{proof}
The proof is done by induction. For $s=1$, the result is given by the relation \eqref{eq:ineq2} from Lemma \ref{lemma:norm}.
For any $s>1$, we have from \eqref{eq:ineq2} and the definition of $\left| \cdot \right|_s$
$$\frac{1}{c'_{s}} \|v\|^2_{H^s} \leq C_{s-1}\| v\|^2_{H^{s-1}} - C'_{s-1}\left| v \right|^2_{s-1}+ \left| v \right|^2_s \leq c'_{s} \| v \|^2_{H^s} \ .$$
Then, we get using the induction property \eqref{eq:norm_weight}
$$\frac{1}{c'_{s}} \|v\|^2_{H^s} + \left( \frac{C'_{s-1}}{c'_{s-1}}- C_{s-1}\right)\| v\|^2_{H^{s-1}}\leq  \left| v \right|^2_s \leq c'_{s} \| v \|^2_{H^s}+  \left( C'_{s-1}c'_{s-1}- C_{s-1}\right)\| v\|^2_{H^{s-1}}\ .$$
Finally, choosing the constant $C'_{s-1}$ such that $ \left( C'_{s-1}/c'_{s-1}- C_{s-1}\right) > 0$, we obtain with new constants  
$$ \frac{1}{c_{s}} \| v \|^2_{H^s}\leq \left| v \right|^2_s \leq c_{s}\| v \|^2_{H^s} \,.$$
\end{proof}

We are now able to give an \textit{a priori} bound on a smooth solution of (qKdV).
\begin{proposition}
\label{prop:apriori}
For any integer $s\geq 4$, a smooth solution $v$ of (qKdV) associated with the initial condition $v_0 \in H^s(\R)$ satisfies
\begin{equation}
\label{eq:apriori}
\|v\|_{H^s} \leq C_s( \|v\|_{W^{3,\infty}}) \|v_0\|_{H^s}
\end{equation}
\end{proposition}
\begin{proof}
Let us come back to the equation satisfied by $\phi_k v_k$. Using what we have done previously, the second order terms cancel out and we can rearrange the remaining terms in the following way 
$$\phi_k \partial_t( \phi_k \vol_k) + \partial_x \phi_k \alpha \partial_x \phi_k \alpha \partial_x \vol_k =  \phi^2_k \widetilde R_k(v) \,,$$
where the function $\widetilde R_k$ gathers all the subprincipal terms we are now able to estimate 
$$ \| \widetilde R(v) \|_{L^2} \leq C_k(\|v\|_{W^{3,\infty}}) \left( \| \phi_k v_k \|_{L^2} + \| v \|_{H^k}\right)\,.$$
Taking the inner product with $v_k$, we obtain
$$\frac{1}{2}\frac{d}{d t} \| \phi_k \vol_k \|^2_{L^2}  = \left<\phi_k \widetilde R^k_{\alpha} | \phi_k v_k \right>\,,$$
and then, using Lemma \ref{lemma:norm}, we find some constant $C_k$ such that 
$$ \frac{d}{d t} \| \phi_k \vol_k \|^2_{L^2} \leq C_k(\|v\|_{W^{3,\infty}}) \left( \| \phi_k v_k \|_{L^2} +  \| v\|_{H^k} \right)\| \phi_k v_k\|_{L^2} \,.$$
Now, from the Proposition \ref{prop:equiv}, without renaming the constants for convenience and by summing the last inequalities for $s \geq k \geq 1$, each one being multiplied by the suitable constant, we get 
\begin{equation}
\frac{d}{d t} \left| v \right|^2_s \leq \max_{k}{\left\{C_k(\|v\|_{W^{3,\infty}})\right\}}\| v \|^2_{H^s} 
\end{equation}
Integrating and calling the maximal constant $C$, we get 
$$ \left| v(t,\cdot) \right|^2_s \leq \left| v_0 \right|^2_s + \int^t_{0} C( \|v(\tau, \cdot)\|_{W^{3,\infty}}) \| v(\tau,\cdot) \|^2_{H^s} d\tau \,.$$
Finally, using Proposition \ref{prop:equiv}, we write 
$$ \| v(t,\cdot) \|^2_{H^s} \leq c_s \left( \left| v_0 \right|^2_s + \int^t_{0} C( \|v(\tau, \cdot)\|_{W^{3,\infty}}) \| v(\tau,\cdot) \|^2_{H^s} d\tau \right) \,.$$
We finish the proof by Gronwall's lemma.
\end{proof}

%%%%%%%%%%%%%%%%%%%%%%%%%%%%%%%%%%%%%%%%%%%%%%%%%%%%%%%%%%%%%%%%%%%%%%%%%%%%%%%%%%%%%%%%%%%%%%%%%%%

\section{Existence of a smooth solution}
This section is devoted to the proof of local well-posedness for (qKdV). We first state the existence and uniqueness of a smooth solution to a parabolic regularized equation with regularized initial data. Then, using uniform a priori bounds in large norms on this smooth solution, we take a limit to prove the existence of solutions to (qKdV). Here, we shall adapt a method by Bona \& Smith \cite{Bona-Smith} and prove directly that the convergence occurs in the very space $\sC(0,T;H^s(\R))$ and check uniqueness and continuity of the solution map with respect to the initial data. 

\subsection{Study of a regularized equation}
Let us introduce a {\it small} parameter $\eps > 0$. For now, we consider the regularized parabolic equation
\begin{equation}
\label{eq:regul}
v_t + av_x + \partial_x \alpha \partial_x \alpha \partial_x v + \eps^4 \partial^4_x v = 0\,.
\end{equation}
Let $\chi$ be a function of class $\sC^{\infty}$ such that its Fourier transform is compactly supported and equals $1$ in a neighborhood of the origin and $\eta \in \sC^{\infty}(\R^*_+)$ a non decreasing function with limit $0$ at $0$ that will be specified later in the proof. 
Denote 
$$ \chi_{\eps} = \frac{1}{\eta(\eps)} \chi \left( \frac{\cdot}{\eta(\eps)}\right)\,.$$
Given $v_0 \in H^s(\R)$, we define a regularized initial data by 
\begin{equation}
\label{eq:initialdata}
v_{0,\eps} = \chi_\eps * v_0\,.
\end{equation}
Regarding the existence of a unique solution to \eqref{eq:regul} with initial data $v_{0,\eps}$, we refer to results on analytic semigroups of semi-linear PDEs in \cite{lunardi} (\S $7.3.2$) and \cite{pazy} (\S $8.4$) combined with semigroup techniques in \cite{BDD1d, BDDmultid}. More precisely, in their work on the Euler-Korteweg system, S. Benzoni-Gavage, R. Danchin, S. Descombes used a similar fourth order regularization to prove the local existence of solutions of linear problems with variable coefficients with a time of existence independent of $\eps$. They use properties of the analytic semigroup generated by $\partial^4_x$ and the Duhamel formula to prove the existence and uniqueness by a fixed point method. Their technique can be directly applied to our semi-linear regularized problem \eqref{eq:regul}. Thus Eq. \eqref{eq:regul} has a unique solution belonging to $\sC(0,T;H^{\infty}(\R))$, with a time of existence $T>0$ depending on the initial data $v_0$ and $\varepsilon$.

Then, we consider a sequence of smooth solutions $(v_\eps)_{\eps > 0}$. To take a limit when $\eps$ tends to zero, we have to justify that the time of existence of the solution may be bounded from below independently of $\varepsilon$. We look for uniform a priori bounds on the solution $v_\eps$ using the techniques presented in the above section. We differentiate Eq. \eqref{eq:regul} by respecting the skew-symmetry of the third order term and use a gauge to cancel remainders. After those two operations, we get that, for all $k \geq 0$
\begin{equation}
\label{eq:regul_deriv}
\phi_k \partial_t( \phi_k \vol_k) + \partial_x \phi_k \alpha \partial_x \phi_k \alpha \partial_x \vol_k  + \phi_k \eps^4 \partial^4_x(\phi_k v_k) =\phi_k^2 \widetilde R_k + \eps^4 \phi^2_k R_{\eps}\,,
\end{equation}
where the term $\widetilde R_k$ contains all previous zero order remainders and the term $R_{\eps}$ is a commutator term arising from the fourth order regularization. To deal with this new commutator term, we first prove the following lemma 

\begin{lemma}
\label{lemma:eps4}
Let $Q(v, \partial_x v, \cdots, \partial^{k+2p}_x v)$ be an homogeneous polynomial of weight $k+2p \in  [k, 2k)$ with terms multiplied by bounded functions of $v$. There exists a constant $\mu >0$ such that for all $v \in H^k(\R)$ and for all small $\eps >0$,
$$  \eps^{2p} \left< Q(v, \partial_x v, \cdots, \partial^{k+2p}_x v)| \partial^k_x v\right> \leq  C_{\mu}\|v\|^2_{H^k} + \eps^{2p}\mu \|\partial^{k+p}_x v\|^2_{L^2}  \,.$$
\end{lemma}
\begin{proof}
The proof uses that of Lemma \ref{lemma:hk}. A general form of Q is 
$$ Q(v, \partial_x v, \cdots, \partial^{k+2p}_x v)  = \sum_i \beta_i(v) \prod_{j=1}^k \partial^{\gamma_j}_x v \ ,  $$ 
where the $\beta_i$ are bounded on $J$ and for all $1 \leq j \leq k$, $\gamma_j  \leq k+2p$ satisfy $ \sum_{j=1}^k \gamma_j = k+2p$.
First, we take the inner product with $\partial^k_x v$ and use an integration by parts to get an expression of the form 
$$ \left< P(v, \partial_x v, \cdots, \partial^{k+p}_x v) | \partial^{k+p}_x v \right>$$
Now, thanks to the triangular inequality, it is sufficient to work on a monomial expression. Without changing notations for convenience, we consider 
$$P(v, \partial_x v, \cdots, \partial^{k+p}_x v)  = \prod_{j=1}^k \partial^{\gamma_j}_x v \,.$$
where for all $1 \leq j \leq k$, $\gamma_j  \leq k+p$ satisfy $ \sum_{j=1}^k \gamma_j = k+p$.

If for all $j$, $\gamma_j \leq k$, we proceed exactly as in Lemma \ref{lemma:hk} to get
$$  \left< P | \partial^{k+p}_x v \right> \lesssim \| v \|_{H^k} \| \partial^{k+p}_x v \|_{L^2}\,.$$ 
Now, if there is a factor with $k < \gamma_{l} \leq k+p$, we use the again Gagliardo-Nirenberg inequality. We know there could be only one because $p \leq k/2$. $$\| \partial^{\gamma_{l}}_x v \|_{L^{2}}  \lesssim \|\partial^{k+p}_x v \|^{\theta}_{L^2} \| \partial^{k}_x v\|^{1-\theta}_{L^2} \ , \textrm{where} \  \theta = \frac{\gamma_{l} - k}{p} < 1\,.$$
Finally, combining with previous terms, we have by Young's inequality 
$$ \| P(v, \partial_x v, \cdots, \partial^{k+p}_x v) \|_{L^2} \lesssim \|v \|^{1-\theta}_{H^k} \|\partial^{k+p}_x v \|^{\theta}_{L^2} \lesssim C_{\mu}\|v \|_{H^k}  + \mu \|\partial^{k+p}_x v \|_{L^2} \,.$$
Now, returning to our first polynomial expression  
$$ \left< P(v, \partial_x v, \cdots, \partial^{k+2p}_x v) | \partial^{k}_x v \right> \leq \left(C_{\mu}\|v \|_{H^k}  + \mu\|\partial^{k+p}_x v \|_{L^2} \right) \|\partial^{k+p}_x v \|_{L^2}\,.$$
Using another time Young's inequality, we find new constants $\mu'$ and $C_{\mu'}$ such that 
$$ \left< P(v, \partial_x v, \cdots, \partial^{k+2p}_x v) | \partial^{k}_x v \right> \leq C_{\mu'}\|v \|^2_{H^k}  + \mu' \|\partial^{k+p}_x v \|^2_{L^2}\,.$$
To conclude the proof, we multiply by the bounded factor $\eps^{2p}$ and find the constant $\mu$ to be the maximum of the $\mu'$ obtained for the various monomials $P$. 
\end{proof}

Let us now come back to the regularized equation \eqref{eq:regul}, and prove the following proposition.
\begin{proposition}
\label{prop:estimates}
Let $v_0 \in H^{q}(\R)$ with $q \geq 4$. For $s\geq q$, the unique solution of \eqref{eq:regul} with regularized initial data $v_{0,\eps}$ defined in \eqref{eq:initialdata} satisfies
$$ \| v(t,\cdot) \|_{H^s} \lesssim  \frac{\|v_0\|_{H^{q}}}{\eta(\eps)^{s-q}}\,.$$
\end{proposition}
\begin{proof}
We follow the steps of proof of Proposition \ref{prop:apriori}. We first take the inner product of \eqref{eq:regul_deriv} with $v_k$ to obtain 
$$\frac{1}{2}\frac{d}{d t} \| \phi_k \vol_k \|^2_{L^2}  + \| \eps^2\partial^{2}_x (\phi_k v_k) \|^2_{L^2}= \left< \phi_k \widetilde R | \phi_k v_k \right> +  \eps^4\left<  \phi_k R_\eps| \phi_k v_k\right> \,.$$
Using the techniques of the previous section and the result of lemma \ref{lemma:eps4}, we obtain by choosing the Young's inequality constant such that $\mu < 1/2$, 
$$\frac{1}{2}\frac{d}{d t} \| \phi_k \vol_k \|^2_{L^2}  + \frac12 \| \eps^2\partial^{2}_x (\phi_k v_k) \|^2_{L^2}  \leq C_{\infty} \| \phi_k v_k \|^2_{L^2} +  C_{k}(\|v\|_{L^{\infty}},\|v_1\|_{L^{\infty}}) \| v\|_{H_k} \| \phi_k v_k\|_{L^2} \,.$$
Now, summing on $1 \leq k \leq s$ and multiplying by the suitable constants $C_k$ at each step, 
$$ \frac{d}{d t} \left| v \right|^2_s + c \sum^s_{k = 1} \| \eps^2\partial^{2}_x (\phi_k v_k) \|^2_{L^2} \leq \max_{k}{C_k} \ \| v \|^2_{H^s}\,.$$
By integration with respect to time
$$ \left|v(t,\cdot) \right|^2_s + \int^t_0 c \sum^s_{k = 1} \| \eps^2\partial^{2}_x (\phi_k v_k) \|^2_{L^2} d\tau \leq \left| v_0 \right|^2_s + \int^t_{0} C\| v \|^2_{H^s} d\tau \ .$$
Finally, 
$$\| v(t,\cdot) \|^2_{H^s} +  c_s\int^t_0 c \sum^s_{k = 1} \| \eps^2\partial^{2}_x (\phi_k v_k) \|^2_{L^2} d\tau\leq c_s \left( \left| v_0 \right|^2_s + \int^t_{0}C \| v\|^2_{H^s} d\tau \right)\ .$$
This last estimate and the norm equivalence give in particular 
$$ \| v(t,\cdot) \|^2_{H^s} \lesssim  \|v_0\|_{H^s} + \int^t_{0} C\| v\|^2_{H^s} d\tau \,.$$
A classical mollifier property from appendix C in \cite{BDDmultid} gives us 
$$\|v_0\|_{H^s} \leq C\frac{\|v_0\|_{H^{q}}}{\eta(\eps)^{s-q}}\,. $$
We finish the proof by Gronwall's lemma. 
\end{proof}
In particular, for initial data $v_0 \in H^s(\R)$, we get a uniform bound of the solution in $H^s(\R)$ for any $s\geq 4$, that is independent of $\eps$. In the proof by S. Benzoni-Gavage, R. Danchin, S. Descombes in \cite{BDD1d}, this uniform estimate is actually used directly in the fixed point argument to justify that the time of existence of the solution $v_\eps$ is independent of the regularization parameter $\varepsilon$. Here we obtain this uniformity a posteriori. From now on, we denote by $T>0$ the minimal common time of existence of all the solutions in the sequence $(v_\eps)_{\eps>0}$ depending only on $v_0$.

\subsection{Convergence to a solution of (qKdV)}
From our regularized equations, we have a sequence of solutions $(v_\eps)_{\eps > 0}$ belonging to $\sC(0,T;H^{\infty}(\R))$ for some $T>0$ given the same initial data $v_{0,\eps}$ regularized from $v_0 \in H^{q}(\R)$, $q \geq4$ for all the sequence.
We shall prove that this sequence is a Cauchy sequence in $\sC(0,T;H^s(\R))$ for any $s\geq  q \geq 4$. 

For $0 < \delta \leq \eps$, we denote by $v_\eps$ and $v_\delta$ the two corresponding solutions of \eqref{eq:regul} and we look for estimates on $z = v_{\eps} - v_\delta$. Then, our goal is to prove that $\|z\|_{L^{\infty}(0,T;H^s(\R))}$ goes to zero when $\eps$ and $\delta$ go to zero. We compute the difference between the two equations on $v_{\eps}$ and $v_\delta$ to find 
\begin{equation*}
\begin{array}{rcl}
z_t + a(v_\delta)z_x + \partial_x \alpha(v_\delta)\partial_x \alpha(v_\delta)\partial_x z + \delta^4 \partial^4_x z &=& \left( \eps^4 - \delta^4\right) \partial^4_x v_{\eps} \\[1ex] 
& + & \left( a(v_\delta) - a(v_\eps) \right)\partial_x  v_\eps\\[1ex]
&+&  \partial_x \alpha(v_\delta)\partial_x \alpha(v_\delta)\partial_x v_\eps \\[1ex]
&- &  \partial_x \alpha(v_\eps)\partial_x \alpha(v_\eps)\partial_x v_\eps\,.
\end{array}
\end{equation*}
We rewrite it in a more compact way
\begin{equation}
\label{eq:diff}
 z_t + \widetilde a_{\eps,\delta} z_x + \partial_x\left(\frac12 \alpha^2_\eps \right) z_{xx} + \partial_x \alpha_\delta \partial_x \alpha_\delta \partial_x z + \delta^4 \partial^4_x z = \left( \eps^4 - \delta^4 \right) \partial^4_x v_\eps + F_{\eps,\delta}(z)\,,
 \end{equation} 
with obvious notations. $F_{\eps,\delta}$ is a linear function with respect to $z$ of homogeneous weight $3$ and 
$$\widetilde a_{\eps,\delta}= \widetilde a_{\eps,\delta}(v_\eps, \partial_x v_\eps, \partial^2_x v_\eps, v_\delta, \partial_x v_\delta, \partial^2_x v_\delta)\,.$$ 
In this last formulation, we gathered all the subprincipal terms. Again, the first order one can be estimated by a direct computation since $\partial_x \widetilde a_{\eps,\delta} \in L^{\infty}(\R)$ according to the estimate on the solutions $v_\eps$ and $v_\delta$. Moreover, to take the limit, we will need estimates on low derivatives of the difference $z$ to compensate the loss involved by the high derivatives in $v_\eps$. More precisely, the arising of terms with too many derivatives on the coefficients $v_\eps$ forces us to use estimates given by Proposition \ref{prop:estimates} and then concede an inverse factor of $\eps$. To recover this factor we shall prove that a low number of derivatives on the difference $z = v_\eps - v_\delta$ can compensate this loss in $\eps$. Then, we shall prove the following lemma
\begin{lemma} 
\label{lemma:estimates_p}
For $0< \delta \leq \eps$, let $v_\eps$ (respectively $v_\delta$) denote the smooth solutions of \eqref{eq:regul} with parameter $\eps$ (respectively $\delta$) and regularized initial data $v_{0,\eps}$ (respectively $v_{0,\delta}$) with $v_0 \in H^{q}(\R)$, $q \geq 4$. Then, for all $0\leq p \leq q$, 
$$\| \partial^p_x \left( v_\eps - v_\delta \right)\|_{L^{\infty}(0,T; L^2(\R))} = o(\eta(\eps)^{q-p})$$
when $\eps$ goes to zero. 
\end{lemma}
\begin{proof}
To get these new estimates, we start by looking for an estimate in $L^2(\R)$ and then in $H^{q}(\R)$. The structure of Eq. \eqref{eq:diff} is obviously different from the one we have worked with previously. In fact, it is principally the arising of the second order term $\partial_x\left(\frac12 \alpha^2_\eps \right) z_{xx}$ which causes troubles. Again, we use a {\it gauge} $\phi_{\eps, \delta}$ to deal with this term and multiply the equation by $\phi^2_{\eps, \delta}$. Exactly as before, by computing commutators we find 
\begin{equation}
\label{eq:regul_phi_l2}
\begin{array}{rcl}
\phi_{\eps, \delta} \partial_t ( \phi_{\eps, \delta} z)  &+& \phi_{\eps, \delta} \widetilde a_{\eps, \delta} \partial_x (\phi_{\eps, \delta}  z )\   + \  \partial_x \alpha_\delta\phi_{\eps, \delta} \partial_x \alpha_\delta\phi_{\eps, \delta} \partial_x z + \delta^4\phi_{\eps, \delta} \partial^4_x \left( \phi_{\eps, \delta} z\right) \\[1ex]
& =&\phi^2_{\eps, \delta} \left( \eps^4 - \delta^4 \right) \partial^4_x v_\eps \  + \ \phi^2_{\eps, \delta} \widetilde F_{\eps, \delta}(z) + \  \delta^4 \phi_{\eps, \delta} [\phi_{\eps, \delta} , \partial^4_x] z \,.
\end{array}
\end{equation}
In the previous expression, we have already used the cancellation due to our gauge and the remainders of order zero are gathered in the term $\widetilde F_{\eps, \delta}$. As before, the ODE the gauge has to satisfy is 
$$ 3 \alpha^2_\delta \phi_{\eps, \delta} \partial_x \phi_{\eps, \delta}  = -  \phi^2_{\eps, \delta} \partial_x \left( \frac12 \alpha^2_\eps\right)\,.$$
To solve this last equation, we rewrite it as 
\begin{equation}
\label{eq:gauge_eps}
\frac{\partial_x \phi_{\eps, \delta}}{\phi_{\eps, \delta}} = -\frac13 \frac{\partial_x \alpha_\eps}{\alpha_\eps} -  \frac{\partial_x(\alpha^2_\eps) }{6} \left( \frac{1}{\alpha^2_\delta} - \frac{1}{\alpha^2_\eps} \right)\,.
\end{equation}
The first part of the right hand side is directly integrable and causes no trouble and the second part belongs to $L^1(\R)$. This yields that $\phi_{\eps, \delta}$ exists and belongs to $L^{\infty}(\R)$ and thus, from Eq. \eqref{eq:gauge_eps}, we can ensure that for $0 \leq l \leq 3$
$$ \partial^l_x \phi_{\eps, \delta} \in L^{\infty}(\R)\,. $$
Moreover, using the equations satisfied by $v_\eps$ and $v_\delta$ and the estimates we have for both of them, we find by a Gagliardo-Nirenberg inequality 
\begin{equation*}
\begin{array}{rcl}
\| \partial_t \phi_{\eps, \delta} \|_{L^\infty} &\lesssim &C + \left(\eps^4 \| \partial^5_x v_\eps \|_{L^\infty} + \delta^4\| \partial^5_x v_\delta \|_{L^\infty} \right)\,, \\[1ex]
&\lesssim &C + \frac{\eps^4}{\eta(\eps)^{3/2}} +  \frac{\delta^4}{\eta(\delta)^{3/2}}\,,\\[1ex]
&\leq& C\,,
\end{array}
\end{equation*}
for a well chosen function $\eta$. Those properties justify that all the remainders from commutators gathered in $\widetilde F_{\eps, \delta}$ are bounded in $L^2(\R)$ as the gauge and all its derivatives arising in the computation are bounded in $L^{\infty}(\R)$. So the adding of this gauge does not change anything from what was done in the previous section, especially regarding the norm equivalences. 

There are two terms left to control. The first one
$$ \delta^4 \phi_{\eps, \delta} [\phi_{\eps, \delta} , \partial^4_x] z$$
is treated using the same scheme of proof as in Lemma \ref{lemma:eps4}. By the Gagliardo-Nirenberg inequality and the estimates on $(v_\eps)_{\eps>0}$, we get constants $\mu > 0$ and $C_\mu$ such that 
$$\delta^4 \left< \phi_{\eps, \delta} [\phi_{\eps, \delta} , \partial^4_x]z | z \right> \leq \delta^{4} C_{\mu}\|v^n\|^2_{L^2} + \delta^{4}\mu \|\partial^{2}_x z\|^2_{L^2}$$
 For the second one, we write 
$$\left< \phi_{\eps, \delta} \left( \eps^4 - \delta^4\right) \partial^4_x v_\eps | \phi_{\eps, \delta} z \right> \lesssim \eps^4 \|\partial^4_x v_\eps\|_{L^2}\|\phi_{\eps, \delta} z\|_{L^2} \lesssim \frac{\eps^4}{\eta(\eps)^{4-q}}\|v_0\|_{H^{q}}\|\phi_{\eps, \delta}z\|_{L^2} \,.$$
Finally, choosing Young's inequality constant such that $\mu < 1/2$, the estimate we get by taking the inner product of eq. \eqref{eq:regul_phi_l2} with $z$ is 
$$\frac12 \frac{d}{dt} \|\phi_{\eps, \delta} z\|^2_{L^2}  + \frac{\delta^4}{2} \|\partial^2_x (\phi_{\eps, \delta} z)\|^2_{L^2} \leq  C \|\phi_{\eps, \delta} z\|^2_{L^2} +\frac{\eps^4}{\eta(\eps)^{4-q}}\|v_0\|_{H^{q}}\|\phi_{\eps, \delta}z\|_{L^2}\,.$$
An integration in time and Gronwall's lemma yield 
$$ \|\phi_{\eps, \delta} z(t)\|_{L^2} \leq  \|\phi_{\eps, \delta} z(0)\|_{L^2} +  \frac{\eps^4}{\eta(\eps)^{4-q}}\|v_0\|_{H^{q}}\,.$$
A classical property on mollifiers, again from appendix C of \cite{BDDmultid}, and norm equivalences finally give 
$$  \|z\|_{L^{\infty}(0,T;L^2(\R))} = o(\eta(\eps)^{q}) +\frac{\eps^4}{\eta(\eps)^{4-q}}  =  o(\eta(\eps)^{q})\,,$$
for a suitable choice of $\eta$, namely
$$ \eta(\eps) = \eps^\beta\,$$ 
with $\beta<1$.

This last relation yields the desired estimate in $L^2(\R)$. Let us now focus on the property in $H^{q}(\R)$ To do this, we need to differentiate the equation $q$ times using again our differential operator $\alpha(v_\delta)\partial_x $ and the corresponding gauge $\phi_{q}(v_\delta)$ from the previous section.
 We recall that $z$ satisfies 
$$ z_t + \widetilde a_{\eps,\delta} z_x + \partial_x\left(\frac12 \alpha^2_\eps \right) z_{xx} + \partial_x \alpha_\delta \partial_x \alpha_\delta \partial_x z - \delta^4 \partial^4_x z = \left( \eps^4 - \delta^4 \right) \partial^4_x v_\eps + F_{\eps,\delta}(z)\,,$$ 
The total gauge we will use is $\Phi_{q}= \phi_{q}(v_\delta)\phi_{\eps, \delta}$. Then, exactly as before, differentiating $q$ times, multiplying by $\Phi^2_{q}$ and computing commutators 
\begin{equation}
\label{eq:regul_deriv}
\begin{array}{rcl}
\Phi_{q} \partial_t ( \Phi_{q} z_{q})  &+& \Phi_{q} \widetilde a_{\eps,\delta} \partial_x (\Phi_{q}  z_{q} )\   + \  \partial_x \alpha_\delta \Phi_{q} \partial_x \alpha_\delta \Phi_{q} \partial_x z_{q} + \delta^4\Phi_{q} \partial^4_x \Phi_{q} z_{q} \\[1ex]
& =& \ \Phi^2_{q} \widetilde F_{\eps,\delta}(z, \cdots , z_{q}) + \  \delta^4 \Phi_{q} [\Phi_{q} , \partial^4_x] z_{q} \ + \ \delta^4 (\Phi_{q})^2 [\left(\partial_x(\alpha_\delta \cdot )\right)^{q} , \partial^4_x] z  \\[1ex]
& & + \  \Phi^2_{q} \left( \eps^4 - \delta^4 \right) \left(\partial_x(\alpha_\delta \cdot)\right)^{q} \partial^4_x v_\eps\,.
\end{array}
\end{equation}
In this last expression, we have already used cancellations from the gauges and the term $\widetilde F_{\eps, \delta}$ contains all the remainders of commutators we have encountered before. This term involves up to $q + 3$ derivatives on $v_\eps$ and up to $q$ derivatives on $v_\delta$ and will have to be estimated again to make sure it remains bounded uniformly in $\eps$. In the following we deal with the four terms in the right hand side.

Let us deal first with the second one. This term is the same as in the $L^2$ case treated previously and we will later set Young's inequality constant to control it. 

For the third one, we use directly the lemma \ref{lemma:eps4} with $k = 4$ and $p = 2$ to get that there exists constants $\mu > 0$ and $C_\mu$ such that 
$$\delta^4 \left< [\left(\partial_x(\alpha_\delta \cdot)\right)^{q} , \partial^4_x] z | z_{q} \right> \lesssim  C_{\mu}\|z\|^2_{H^q} + \delta^4\mu \|\partial^{2}_x z_{q}\|^2_{L^2}\,.$$

We rewrite the last term as
$$\Phi^2_{q} \left( \eps^4 - \delta^4 \right) \left(\partial_x((\alpha_\delta - \alpha_\eps) \cdot)\right)^{q} \partial^4_x v_\eps + \Phi^2_{q} \left( \eps^4 - \delta^4 \right) \left(\partial_x(\alpha_\eps \cdot)\right)^{q} \partial^4_x v_\eps\,.$$
In the first part, it appears at most order $q$ derivatives of $z$ and order $q + 4$ derivatives of $v_\eps$. We rewrite it as a sum of terms of the general form 
$$\left( \eps^4 - \delta^4 \right) \partial^q_x z \ \partial^{q+4-l}_x v_\eps\,, $$
with $0\leq l \leq q$. Using a bootstrap argument, we get the estimate 
$$\left( \eps^4 - \delta^4 \right) \|\partial^l_x z \ \partial^{q+4-l}_x v_\eps \|_{L^2} \lesssim \eps^4 \|\partial^q_x z\|_{L^2}  \|\partial^{q+4-l}_x v_\eps \|_{L^\infty} = \frac{o(\eta(\eps)^{q-l}) \ \eps^4}{\eta(\eps)^{(4-l)/2 + (5-l)/2}} = o(1)\,.$$
For the second part of this last term, the term which has the worst possible loss of derivative is
$$\left( \eps^4 - \delta^4 \right) v_\eps \ \partial^{q+4}_x v_\eps\,.$$
In this case, we estimate as before 
$$\left( \eps^4 - \delta^4 \right) \|v_\eps \ \partial^{q+4}_x v_\eps\|_{L^2} \lesssim \frac{\eps^4}{\eta(\eps)^{4}} = o(1) \,,$$
with our definition of the function $\eta$. This final estimate proves that the last term in eq. \eqref{eq:regul_deriv} goes to zero when $\eps$ goes to zero. 

We are left with $\widetilde F_{\eps,\delta}$, which we rewrite in general form 
$$ \sum_i \beta_i(v_\eps, v_\delta,z) \prod_{j,l,m} \partial^{\gamma_j}_x v_\eps \partial^{\gamma_l}_x v_\delta \partial^{\gamma_m}_x z\,,$$
where for all $j$, $l$ and $m$, $\gamma_j \leq q + 3$, $\gamma_l \leq q$, $\gamma_m \leq q$ and 
$$\sum_{j, l,m}  \gamma_j + \gamma_l + \gamma_m = q + 3\,.$$ 
If $\gamma_j \leq q$ then the Gagliardo-Nirenberg inequality gives us our estimate as $3 \leq \gamma_m \leq q$ and there is no loss in $1/\eta(\eps)$ involved. The issue occurs when there is more than $q$ derivatives on either $v_\eps$. In this case, $\gamma_l \leq 3$ and the corresponding factors are bounded in $L^{\infty}$. Then, we write for $1 \leq l \leq 3$ 
$$ \|\partial^{q + 3 - l}_x v_\eps\|_{L^\infty} \|\partial^l_x z\|_{L^2} = \frac{o(\eta(\eps)^{q - l})}{\eta(\eps)^{3+1/2-l}} = o(1)\,.$$

Finally, let us gather all we have done before. We take the inner product of Eq. \eqref{eq:regul_deriv} by $z_{q}$ and, gathering all Young's inequalities constants such that their sum is less than $1/2$, we can write
$$  \frac12 \frac{d}{dt} \|\Phi_{q} z_{q}\|^2_{L^2}  + \frac{\delta^4}{2} \|\partial^2_x (\Phi_{q} z_{q})\|^2_{L^2} \lesssim \|\Phi_{q} z_{q}\|^2_{L^2} +  F(\eps)\|\Phi_{q} z_{q}||_{L^2}\,, $$
where we gathered in $F(\eps)=o(1)$ all the previously treated terms. Gronwall's lemma yields 
\begin{equation}
\label{eq:bs_estimate}
\|\Phi_{q} z_{q}(t,\cdot)\|_{L^2} \lesssim \|\Phi_{q} z_{q}(0, \cdot )\|_{L^2} + F(\eps)\,,
\end{equation}
and with mollifiers properties 
$$\|\Phi_{q} z_{q}\|_{L^{\infty}(0,T;L^2(\R))}= o(1)\,.$$
 Together with $\|\phi_{\eps,\delta} z\|^2_{L^{\infty}(0,T;L^2(\R))} = o(\eta(\eps)^{q})$, we complete the proof by interpolation and norm equivalences. 
\end{proof}

\begin{corollary}
\label{cor:existence}
The sequence $(v_\eps)_{\eps>0}$ is a Cauchy sequence in $\sC(0,T;H^s(\R))$ for any $s > 3+1/2$. Then, its limit $v \in \sC(0,T;H^s(\R))\cap \sC^1(0,T;H^{s-3}(\R))$ is a solution to (qKdV) with initial data $v_0$. 
\end{corollary}

%%%%%%%%%%%%%%%%%%%%%%%%%%%%%%%%%%%%%%%%%%%%%%%%%%%%%%%%%%%%%%%%%%%%%%%%%%%%%%%%%%%%%%%%%%%%%%%%%%%%%%%%

\section{Uniqueness and continuity with respect to the data}
As announced, here we adapt a technique originally introduced by Bona \& Smith in \cite{Bona-Smith} and later exploited in many papers, see \cite{Bburgers, BDD1d, BDDmultid} for example. We prove the following theorem 
\begin{theorem}
For an integer $s \geq  4$, let $K$ be a strictly positive constant. For all $v_0 \in H^s(\R)$ of norm not greater than $K$, the mapping 
\begin{equation*}
\label{eq:mapping}
\begin{array}{rcl}
H^s(\R) &\to &\sC(0,T;H^s(\R)) \cap \sC^1(0,T, H^{s-3}(\R))\\
v_0 &\mapsto &v \,, \quad \text{ solution of (qKdV) with initial data } v_0\,
\end{array}
\end{equation*}
is continuous. 
\end{theorem}
\begin{proof}
We aim at proving that for any sequence of initial conditions $(v^n_0)_{n \geq 0}$ going to $v_0$ in $H^s(\R)$, then the corresponding sequence of solutions $(v^n)_{n \geq 0}$ goes to $v$, the solution corresponding to the initial data $v_0$. We start by writing 
$$ \|v^n- v\|_{H^s} \leq \|v^n - v^n_{\eps}\|_{H^s} + \|v^n_\eps - v_\eps\|_{H^s}+ \|v_{\eps} - v\|_{H^s}\,.$$
 We first focus on the first and third terms. Let us rewrite what we obtained in \eqref{eq:bs_estimate}. For $\eps \geq \delta > 0$, using norm equivalences and taking the limit $\delta \to 0$, we get
\begin{equation}
\| v_\eps(t,\cdot) - v(t,\cdot) \|_{H^s} \leq  C_K \left(\| v_{\eps,0} - v_0 \|_{H^s} + F(\eps) \right)\,,
\end{equation}
where $F(\eps)$ goes to zero when $\eps$ goes to zero. This kind of estimate is also true for the difference between the solutions $v^n_\eps$ and $v^n$
\begin{equation}
\| v^n_\eps(t,\cdot) - v^n(t,\cdot) \|_{H^s} \leq  C_K \left(\| v^n_{\eps,0} - v^n_0 \|_{H^s} + F(\eps) \right)\,.
\end{equation}
Moreover, we have  
\begin{equation}
\begin{array}{rcl}
\| v^n_{\eps,0} - v^n_0 \|_{H^s} & \leq& \| v^n_{\eps,0} - v_{\eps,0} \|_{H^s} + \|v_{\eps,0} - v_0 \|_{H^s}  + \| v_0 - v^n_0 \|_{H^s} \,,\\[1ex]
& \leq & 2\|v^n_0 - v_0\|_{H^s} +  \|v_{\eps,0} - v_0 \|_{H^s}\,.
\end{array}
\end{equation}
Now for the second term, we have to revisit the proof of Proposition \ref{lemma:estimates_p} and more precisely the way we obtained estimate \eqref{eq:bs_estimate}. Here, we have to estimate the difference between two solutions with same regularization parameter $\eps$, but which satisfy the same regularized equation with different initial data. We proceed exactly the same way with some cancellations due to the fact that we actually take $\delta = \eps$ in the computation. The only point where we cannot follow the proof concerns the terms evaluated by  
$$ \|\partial^{s + 3 - q}_x v_\eps\|_{L^\infty} \|\partial^q_x (v^n_\eps-v_\eps )\|_{L^2}\,,$$
for $0 \leq q \leq 3$. 
Indeed, we have balanced the coefficients involving $\eps$ to actually get uniform estimates with respect to this parameter (it was the purpose of Lemma \ref{lemma:estimates_p}) but here, we cannot do the same. In fact we cannot hope to have a uniform estimate in $\eps$ but, by a Gagliardo-Nirenberg inequality we find 
$$ \|\partial^{s + 3 - q}_x v_\eps\|_{L^\infty} \|\partial^q_x  (v^n_\eps-v_\eps )\|_{L^2}  \leq \frac{C}{\eta(\eps)^{3 - q + 1/2}} \|v^n_\eps-v_\eps\|_{H^s}\,.$$
Finally, we obtain the following estimate 
\begin{equation}
\|v^n_\eps - v_\eps\|_{L^{\infty}H^s} \leq C_\eps \|v^n_0 - v_0\|_{H^s}\,,
\end{equation}
with the constant $C_\eps$ going to $+ \infty$ when $\eps$ goes to zero. Now, using all previous estimates 
$$\|v^n- v\|_{L^{\infty}H^s} \leq  C_K \left(\| v^n_{\eps,0} - v^n_0 \|_{H^s} + F(\eps) \right) + C_\eps \|v^n_0 - v_0\|_{H^s} +C_K \left(\| v_{\eps,0} - v_0 \|_{H^s} + F(\eps) \right)\,,$$
which finally yields 
$$\|v^n- v\|_{L^{\infty}H^s} \leq 2 C_K \left(\| v_{\eps,0} - v_0 \|_{H^s}+F(\eps) \right) + (2C_K + C_\eps) \|v^n_0 - v_0\|_{H^s}\,.$$
Now we find that 
$$\limsup_{n \to +\infty } \|v^n- v\|_{L^{\infty}H^s} \leq 2 C_K \left(\| v_{\eps,0} - v_0 \|_{H^s}+F(\eps) \right)\,.$$
This finishes the proof as the right hand side of the last inequality goes to zero when $\eps$ goes to zero for any initial condition $v_0$. 
\end{proof}
The proof of uniqueness is a straightforward corollary of the previous one. Instead of considering the difference between $v^n$ and $v$, we consider the difference between two different solutions $u$ and $v$ but with the same initial data $u_0 = v_0$. Following the exact same steps of the previous proof, with $ v_{\eps,0} =  u_{\eps,0}$, we deduce 
$$\|u- v\|_{L^{\infty}H^s} \leq 2 C_K \left(\| v_{\eps,0} - v_0 \|_{H^s}+F(\eps) \right)\,.$$
Then, we get uniqueness when taking the limit $\eps$ goes to zero. 

\section{Concluding remarks}

The present well-posedness result is set on the real line $x \in \R$. Nevertheless, as the proof use only the structure of the qKdV equation and does not uses any dispersion estimate, nothing prevents us from considering the same problem set on the torus $\R/\Xi\Z$ for any period $\Xi >0$. Then, we are able to give a similar well-posedness result on a unidimensional torus.  

\begin{theorem}\label{thm:qKdV}
Assume that $k\geq 4$. If $p=-f': I \subset \R \to \R$  is $\sC^{k+1}$ and $\capp:I \to \R^{+*}$ is $\sC^{k+2}$, then for all $\Xi > 0$, $\bv_0\in H^k(\R/\Xi\Z)$, there exists a time $T>0$ and a unique $\bv\in \sC(0,T;H^k(\R/\Xi\Z)) \cap \sC^1(0,T;H^{k-3}(\R/\Xi\Z))$ solution to \eqref{eq:qkdv} with initial data $v_0$. Moreover, $\bv_0\mapsto \bv$ maps continuously $H^k(\R/\Xi\Z)$ into $\sC(0,T;H^k(\R/\Xi\Z)) \cap \sC^1(0,T;H^{k-3}(\R/\Xi\Z))$.
\end{theorem}

%\begin{theorem}\label{thm:EKE}
%If $F: (0,+\infty) \to \R$ and $\Cap:(0,+\infty) \to (0,+\infty)$ are ${\sC$, for all $\Xi>0$, $s>3/2$, 
%$(\rho_0,\vits_0)\in H^{s+1}(\R/\Xi\Z)\times H^{s}(\R/\Xi\Z)$, $\rho_0>0$, there exists $T>0$ and a unique $(\rho,\vits)\in {\class}([0,T);H^{s+1}(\R/\Xi\Z)\times H^{s}(\R/\Xi\Z))$ solution to {\rm (EKE)} with $\En=F(\rho)+\frac12 \Cap(\rho)\rho_x^2$, and 
%$(\rho_0,\vits_0)\mapsto (\rho,\vits)$ is continuous.
%\end{theorem}
Though we have not carried out such task here, it is likely that one may adapt the strategy expounded here to deal with non integer indices of regularity and relax the constraint $k\geq4$ to $k>3+1/2$.

At first, we were aiming at a well-posedness result compatible with the study of the non-linear stability of a known solution. More precisely, we were initially looking for a solution of the equation as a perturbation around a bounded and infinitely differentiable given solution of (qKdV). Nevertheless, it appears that the gauge technique does not work well in this case. The functions defining the gauge are harder to construct when coefficients depend not only the unknown solution but also on the known profile, especially when both parts have different localization properties. A perspective we did not investigate here would be the study of a perturbations around a periodic state. The periodicity property, used similarly as in \cite{RZ15}, should enable us to overcome the previous difficulty and should give us the suitable gauge estimates.

%Another remark is that we have not been able to prove the existence of solutions by an iterative method here. In fact, it is again the gauge technique that fails in this context. When trying to find estimates at index $n$, we face terms involving derivatives of the unknown but at a previous index, which cannot be cancelled by a gauge constructed on the unknown of index $n$. It indicates that the gauge technique is not the best suited for the existence theory for lower regularities than the one prescribed by the energy estimate. 

\paragraph{Acknowledgement.} The author would like to thank S. Benzoni-Gavage and L.M. Rodrigues for their helpful guidance through many fruitful discussions. This work has been supported by ANR project BoND (ANR-13-BS01-0009-01). The author also wants to thank Fr\'ed\'eric Rousset for several insightful comments made during some of the BoND meetings.

\newpage 
\bibliographystyle{plain}
\addcontentsline{toc}{section}{References}
\bibliography{Mie.bib}

\end{document}